\documentclass[12pt]{amsart}

\usepackage{amsmath,amsthm,amssymb}
\usepackage{amsfonts}
\usepackage{comment}
\usepackage{amsrefs}
\usepackage{hyperref}
\usepackage{mathtools}
\usepackage{bm}
\usepackage[margin=1in]{geometry}

\usepackage{mathrsfs}
\usepackage{enumerate}

\usepackage{todonotes}

\newtheorem{theorem}{Theorem}[section]

\newtheorem{lemma}[theorem]{Lemma}

\theoremstyle{definition}

\newtheorem*{theorem*}{Theorem}
\newtheorem*{proposition*}{Proposition}
\newtheorem*{lemma*}{Lemma}

\theoremstyle{remark}
\newtheorem*{remark}{Remark}

\numberwithin{equation}{section}

\newcommand{\QQ}{\mathbb Q}
\newcommand{\OO}{\mathcal O}

\newcommand{\NN}{\mathbb N}
\newcommand{\RR}{\mathbb R}

\makeatletter
\newcommand*{\defeq}{\mathrel{\rlap{%
                     \raisebox{0.3ex}{$\m@th\cdot$}}%
                     \raisebox{-0.3ex}{$\m@th\cdot$}}%
                     =}
\makeatother
\renewcommand{\aa}{\mathfrak a}

\newcommand{\PP}{\mathbb P}
\newcommand{\II}{\mathcal I}

\usepackage{mathrsfs}

\author[Evan Chen]{Evan Chen}
\address{Department of Mathematics, Massachusetts Institute of Technology, \mbox{Cambridge, MA 02139}}
\email{\href{mailto:evanchen@mit.edu}{{\tt evanchen@mit.edu}}}

\author[Peter S. Park]{Peter S. Park}
\address{Department of Mathematics, Princeton University, Princeton, NJ 08544}
\email{\href{mailto:pspark@math.princeton.edu}{{\tt pspark@math.princeton.edu}}}

\author[Ashvin A. Swaminathan]{Ashvin A. Swaminathan}
\address{Department of Mathematics, Harvard College, \mbox{Cambridge, MA 02138}}
\email{\href{mailto:aaswaminathan@college.harvard.edu}{{\tt aaswaminathan@college.harvard.edu}}}

\begin{document}

\title[On Logarithmically Benford Sequences]{On Logarithmically Benford Sequences}
\date{\today}

\begin{abstract}
Let $\II \subset \NN$ be an infinite subset, and let $\{a_i\}_{i \in \II}$ be a sequence of nonzero real numbers indexed by $\II$ such that there exist positive constants $m, C_1$ for which $|a_i| \leq C_1 \cdot i^m$ for all $i \in \mathcal{I}$. Furthermore, let $c_i \in [-1,1]$ be defined by $c_i = \frac{a_i}{C_1 \cdot i^m}$ for each $i \in \II$, and suppose the $c_i$'s are equidistributed in $[-1,1]$ with respect to a continuous, symmetric probability measure $\mu$. In this paper, we show that if $\II \subset \NN$ is not too sparse, then the sequence $\{a_i\}_{i \in \II}$ fails to obey Benford's Law with respect to arithmetic density in any sufficiently large base, and in fact in any base when $\mu([0,t])$ is a strictly convex function of $t \in (0,1)$. Nonetheless, we also provide conditions on the density of $\II \subset \NN$ under which the sequence $\{a_i\}_{i \in \II}$ satisfies Benford's Law with respect to logarithmic density in every base.

As an application, we apply our general result to study Benford's Law-type behavior in the leading digits of Frobenius traces of newforms of positive, even weight. Our methods of proof build on the work of Jameson, Thorner, and Ye, who studied the particular case of newforms without complex multiplication.
\end{abstract}
\maketitle

\section{Introduction}
It was first noted in 1881 by astronomer Simon Newcomb that when numbers occur in the real world, their leading digits tend not to be uniformly distributed. Specifically, Newcomb observed while studying tables of logarithms that certain pages were more worn away than others, especially those pages corresponding to logarithms whose first digit is $1$~\cite{book}. In 1938, physicist Frank Benford corroborated this hypothesis in a considerably more general setting by testing it on an extensive data set including population sizes, physical constants, molecular weights, and even the surface areas of rivers~\cite{benford}. This bias toward certain initial digits, which is known as \emph{Benford's Law}, has since been discovered to hold for a number of distributions that arise in modern mathematics (see~\cite{book} for an informative exposition on the subject). But before we discuss specific examples of sequences that obey Benford's Law, we must pause to state the law in a precise and general manner.

\subsection{Definitions}

Let $\NN$ denote the set of positive integers, and let $\II \subset \NN$ be an infinite subset. Given an infinite sequence $\aa = \{a_i\}_{i \in \II}$ of nonzero real numbers indexed by $\II$ and a subset $A \subset \RR$, we can associate to the pair $(\aa, A)$ an \textit{arithmetic density} $d(\aa, A)$ \mbox{that is given by}
\begin{equation}\label{imwideawake}
	d(\aa,A)
	\defeq
	\lim_{x \to \infty}
	\frac%
	{ \#\{i \le x : i \in \II \text{ and } a_i \in A\} }
	{ \#\{i \le x : i \in \II\} },
\end{equation}
if the limit exists. The above definition serves as a means of quantifying the density of elements of the set $A$ in the sequence $\aa = \{a_i\}_{i \in \II}$. We say that $\aa$ is \emph{(arithmetically) Benford} in base $b$ if for any (nonzero) string of base-$b$ digits $S_b$, we have that
\[ d(\aa,b,S_b) \defeq d\big(\aa,\{ x \in \RR : \text{$|x|$ begins with $S_b$ in base $b$} \}\big) = \log_b ( 1 + S_b^{-1}), \]
where in computing the logarithm, we interpret $S_b$ as an integer. For example, if $\aa$ is Benford in base $10$, then the terms $a_i$ start with the digit $1$ about $d(\aa, 10, 1) \approx 30\%$ of the time and with the digit $9$ about $d(\aa, 10, 9) \approx 4.5\%$ of the time.

A number of sequences of arithmetic interest, such as the sequence of factorials and the partition function, are Benford in any base $b \geq 2$. Benford's law has also been proven for the distribution of values taken by $L$-functions~\cite{kontorovich}, data from dynamical systems (e.g. linearly-dominated systems and nonautonomous
dynamical systems)~\cite{berger}, and truncated progressions of the $3x+1$ problem~\cites{kontorovich, ksound}.

Nonetheless, there are many natural sequences of numbers that do not satisfy this strong property -- most notably the set of positive integers $\NN$, which fails to be Benford in any base $b \geq 2$. However, it is possible to show that such sequences still demonstrate a Benford-type behavior, as long as we consider a different, more inclusive notion of density.
Given a sequence $\aa = \{a_i\}_{i \in \II}$ and $A$ as before, and letting $\II_{\leq x} = \{i \in \II : i \leq x\}$, we can associate to the pair $(\aa, A)$ a \textit{logarithmic density} $\delta(\aa, A)$ defined as
\begin{equation}\label{wreckingball}
	\delta(\aa, A)
	\defeq
	\lim_{x \to \infty}
	\frac%
	{ \displaystyle\sum_{i \in \II_{\le x},a_i \in A} \tfrac{1}{i}}
	{ \displaystyle\sum_{i \in \II_{\le x}} \tfrac{1}{i}},
\end{equation}
if the limit exists. We then say that the sequence $\aa$ is \emph{logarithmically Benford} in base $b$ if for any (nonzero) string of base-$b$ digits $S_b$, we have that
\[ \delta(\aa,b,S_b) \defeq \delta\big(\aa,\{ x\in \RR : \text{$|x|$ begins with $S_b$ in base $b$} \}\big) = \log_b ( 1 + S_b^{-1}).\]
It is known (e.g., see~\cite{stronger}) that if the arithmetic density $d(\aa,A)$ exists, then the logarithmic density $\delta(\aa, A)$ also exists and equals $d(\aa,A)$.
However, the converse of this statement is false; as stated in~\cite{strongerex}, both the sequence of natural numbers and that of prime numbers are logarithmically Benford with respect to any string $S_b$ in any base $b$. 
Therefore, the condition of being logarithmically Benford
is strictly weaker than that of being arithmetically Benford.
\begin{remark}
There are other types density, such as Dirichlet density, that are strictly weaker than arithmetic density (in fact, Dirichlet density is strictly weaker than logarithmic density). In this paper, we restrict our consideration to arithmetic and logarithmic densities, as these are the most commonly studied. It may however be interesting to determine whether there are sequences of mathematical importance that satisfy Benford's Law with respect to other types of density. For example, it is known that the primes are Benford with respect to logarithmic (and hence Dirichlet) density. See~\cite{pompoms} for a more description of generalized asymptotic densities, and see~\cite{course} for a discussion of Dirichlet density in particular.
\end{remark}

\subsection{Statement of Results}\label{sob}

One interesting occurrence of the logarithmic Benford property lies in the study of Fourier coefficients of certain modular forms, namely newforms (i.e.,~holomorphic cuspidal normalized Hecke eigenforms; see~\cite{ken} for a standard reference). In~\cite{ye}, Jameson, Thorner, and Ye employed the Sato-Tate conjecture to show that for a newform $f$ of even weight \emph{without} complex multiplication, the sequence $\{a_f(p) : p \text{ prime}\}$ of Frobenius traces of $f$ is not arithmetically Benford in any base $b \geq 2$, but is logarithmically Benford in every base. As we show in this paper, the method of proof employed in~\cite{ye} can be modified to yield more general results about when sequences are logarithmically Benford but not arithmetically Benford. Moreover, we prove as a corollary that the theorems of Jameson, Thorner, and Ye~hold for such newforms with complex multiplication as well.

We shall consider sequences $\aa = \{a_i\}_{\in \II}$ whose growth is bounded by a power function of the index $i$; specifically, suppose there exist constants $m, C_1 > 0$ such that $|a_i| \leq C_1 \cdot i^m$ for all $i \in \mathcal{I}$. As noted in~\cite{rolex}, proving that a sequence is Benford usually requires an equidistribution result of some sort as input, so we impose the following assumption on our sequence $\aa$: taking $c_i \in [-1,1]$ to be defined by $c_i = \frac{a_i}{C_1 \cdot i^m}$ for each $i \in \II$, suppose the $c_i$'s are equidistributed in $[-1,1]$ with respect to a continuous, symmetric probability measure $\mu$. Note that this assumption implies that for any $[A,B] \subset [-1,1]$ we have
\begin{equation}\label{sorreee}
\frac{\#\{i \leq x : i \in \II \text{ and } c_i \in [A,B]\}}{\#\{i \leq x : i \in \II\}} = (1 + o(1)) \cdot \mu([A,B])  ,
\end{equation}
from which one readily deduces that
\begin{equation}
\sum_{\substack{i \in \II_{\leq x} \\ c_i \in [A,B]}} \frac{1}{i} = (1 + o(1)) \cdot \mu([A,B]) \cdot \sum_{\substack{i \in \II_{\leq x}}} \frac{1}{i}.\label{imfeelin22}
\end{equation}

We wish to determine whether the sequence $\aa$ is arithmetically, or at least logarithmically, Benford in any base $b \geq 2$. To this end, we prove two main theorems, the first of which is stated as follows.

\begin{theorem}
Retain the above setting, and suppose for every $c > 1$ that $\II \cap [x, cx] \neq \varnothing$ for all sufficiently large $x$. Then $\aa$ is not arithmetically Benford in any sufficiently large base. If in addition $\mu([0,t])$ is a strictly convex function of $t \in (0,1)$, then $\aa$ is not arithmetically Benford in any base $b \ge 2$.
	\label{thm:arithmetic}
\end{theorem}

On the other hand, our second theorem indicates conditions under which the sequence $\aa$ is logarithmically Benford in every base $b \geq 2$.

\begin{theorem}
Retain the above setting, and suppose $\II$ is such that $\sum_{i \in \II_{\leq x}} \frac{1}{i} = (1 + o(1)) \cdot C_2 \cdot g(x)$, where $C_2 > 0$ is a constant and where we may take $g(x) = \log x$ or $g(x) = \log \log x$. Then $\aa$ is logarithmically Benford in every base $b \ge 2$.
	\label{thm:logarithmic}
\end{theorem}

\begin{remark}
We note that the assumptions made in stating the above theorems are reasonable. Indeed, as shown in~\cite{rolex}, certain sequences with more than polynomial growth are known to be arithmetically Benford. Furthermore, the proof of Theorem~\ref{thm:logarithmic} depends heavily on the particular properties of the functions $g(x) = \log x$ and $g(x) = \log \log x$.
\end{remark}

The rest of this paper is organized as follows. Section~\ref{ari} presents the proof of Theorem~\ref{thm:arithmetic}, and Section~\ref{log} discusses the proof of Theorem~\ref{thm:logarithmic}. Section~\ref{sadist} concludes the paper with an application of Theorems~\ref{thm:arithmetic} and~\ref{thm:logarithmic} to studying Benford's Law-type behavior in the sequence of Frobenius traces of a newform. 

\section{Proof of Theorem~\ref{thm:arithmetic}}\label{ari}
To begin with, we observe that the proof of Theorem 1 in~\cite{ye} works, \emph{mutatis mutandis}, to prove that $\aa$ is not arithmetically Benford in any sufficiently large base. Indeed, the only conditions that the proof of~\cite{ye} ever uses are that $\II$ have nontrivial intersection with every sufficiently large subinterval of $\NN$ and that $\mu$ be symmetric and continuous. However, as we will now show, the case where $\mu([0,t])$ is strictly convex has a much cleaner proof.

Suppose $b \ge 3$,
and let $1_b$ be the string of digits whose only character is the digit $1$,
interpreted in the base $b$.
We will prove that when $S_b = 1_b$,
the limit (\ref{imwideawake}) defining the arithmetic density
\[ d\big(\aa, b, 1_b\big) \]
does not exist, which is enough to imply the theorem in this case. Assume for the sake of contradiction that this limit exists,
and for every $n \in \NN$, put
\[ I_n^- \defeq
	\left\{ i \in \II : \frac{40}{23} b^n \le C_1 \cdot i^m < 2 b^n \right\}.  \]
Observe that $I_n^- \neq \varnothing$ for sufficiently large $n$ because of our assumption on the sparseness of $\II \subset \NN$. Suppose that for some $i \in I_n^-$ there exists nonnegative integer $j$ such that
\[ \frac{23}{40} b^{-j} \le \left|c_i\right| \le b^{-j}. \]
Then, we have by the definition of $c_i$ that
\[ b^{n-j} \leq |a_i| < 2 b^{n-j}, \]
from which we readily deduce that $|a_i|$ begins with the digit $1$ in base $b$.
Now, by appealing to Theorem~\ref{thm:satotate}, we have the following lower bound on the desired density:
\begin{eqnarray}
	d\big(\aa, b, 1_b\big)
	& = &
	\lim_{n\to\infty}
	\frac{\#\{i \in I_n^- : a_i \text{ begins with $1_b$}\}}{\#I_n^-} \nonumber \\
	& \ge & \sum_{j = 0}^\infty \lim_{n\to\infty}
	\frac{\#\{i \in I_n^- : \tfrac{23}{40} b^{-j} \le \left|c_i\right| \le b^{-j}\}}{\#I_n^-} \nonumber \\
    & = &
	2\sum_{j = 0}^\infty \mu\left(\left[\frac{23}{40}b^{-j}, b^{-j}\right]\right). \label{eq:Sad}
\end{eqnarray}
We wish to derive an upper bound on $d\big(\aa, b, 1_b\big)$ in a similar manner.
Consider indices $i$ such that $b^n \le |a_i| < 2b^n$,
and define the interval
\[ I_n^+ \defeq \left\{ i \in \II :
\frac52 b^n < C_1 \cdot i^m \le \frac83 b^n \right\}. \]
If $|a_i|$ begins with the digit $1$, and $i \in I_n^+$,
we deduce that $\frac38 b^{-j} \le |c_i| \le \frac45b^{-j}$ for some integer $j$;
but as $b \ge 3$ and $|c_i| \le 1$ this only makes sense for $j \ge 0$.
Thus we obtain an upper bound
\begin{eqnarray}
	d\big(\aa, b, 1_b\big)
	& = &
	\lim_{n\to\infty}
	\frac{\#\{i \in I_n^+ : a_i \text{ begins with $1_b$}\}}{\#I_n^+} \nonumber \\
	& \le &
	2\sum_{j = 0}^\infty\mu\left(\left[ \frac38 b^{-j}, \frac45 b^{-j}\right]\right) \label{eq:Swag}
\end{eqnarray}
Combining \eqref{eq:Sad} and \eqref{eq:Swag}, we get
\[
\sum_{j = 0}^\infty \mu\left(\left[\frac{23}{40}b^{-j}, b^{-j}\right]\right) \le \sum_{j = 0}^\infty\mu\left(\left[ \frac38 b^{-j}, \frac45 b^{-j}\right]\right),
\]
but this contradicts our convexity assumption; we have for all $x \in (0,1)$ that 
\[
\mu\left(\left[\frac{23}{40}x, x \right]\right) > \mu\left(\left[ \frac38 x, \frac45 x\right]\right).
\]

For $b = 2$, a similar argument can be employed,
but we cannot simply take the string $1_b$ since
$d(\mathfrak a, 1, 1_b)$ trivially equals $1 = \log_2(1+1^{-1})$.
Instead, we repeat the above argument using the string $10_b$; the lower bound is obtained by counting $i \in \II$ with
\[ \frac{11}{15} \cdot 4^{-j} \le |c_i| \le 4^{-j} \]
across the interval $\frac{30}{11} \cdot 4^n \le C_1 \cdot i^m < 3 \cdot 4^n$,
whilst the upper bound is obtained by counting $i \in \II$ with
\[ \frac25 \cdot 4^{-j} \le |c_i| \le \frac23 \cdot 4^{-j} \]
across the interval $\frac{9}{2} \cdot 4^n \le C_1 \cdot i^m < 5 \cdot 4^n$.
We can then obtain a similar contradiction by using the convexity assumption on $\mu$ to see that
$ \mu\left( \left[ \frac{11}{15}x, x \right] \right) >
\mu\left( \left[ \frac25x, \frac23x \right] \right)$
for all $x \in (0,1)$.

\section{Proof of Theorem~\ref{thm:logarithmic}}\label{log}

Recall the assumptions of Theorem~\ref{thm:logarithmic}: we take our measure $\mu$ to be symmetric and our index set $\II$ to satisfy $\sum_{i \in \II_{\leq x}} \frac{1}{i} = (1+o(1)) \cdot C_2 \cdot g(x)$, where $C_2 > 0$ is a constant and where $g(x) = \log x$ or $g(x) = \log \log x$. To show that $\aa$ is logarithmically Benford in any base $b \geq 2$, it suffices to show that
\begin{equation}\label{eq:yomama}
\sum_{\substack{i \in {\II}_{\leq x} \\ a_i \in  A(b,S_b)}} \frac{1}{i} = (1+o(1)) \cdot \log_b(1+S_b^{-1}) \cdot (C_2 \cdot g(x)),
\end{equation}
where $A(b,S_b) = \{ x \in \RR : \text{$|x|$ begins with $S_b$ in base $b$} \}$. Let $r \in \NN$, to be specified later, and split the sum on the left-hand-side of~\eqref{eq:yomama} into two pieces, according as $|c_i| \leq \tfrac{1}{r}$ or $|c_i| > \tfrac{1}{r}$:
\begin{equation}\label{dubbleD}
\sum_{\substack{i \in {\II}_{\leq x} \\ a_i \in  A(b,S_b)}} \frac{1}{i} = \sum_{\substack{i \in {\II}_{\leq x} \\ a_i \in  A(b,S_b) \\ |c_i| \leq 1/r}} \frac{1}{i} + \sum_{\substack{i \in {\II}_{\leq x} \\ a_i \in  A(b,S_b) \\ |c_i| > 1/r}} \frac{1}{i}.
\end{equation}
It is fairly straightforward to bound the first sum on the right-hand-side of~\eqref{dubbleD}:
\begin{equation}\label{raiders}
\sum_{\substack{i \in {\II}_{\leq x} \\ a_i \in  A(b,S_b) \\ |c_i| \leq 1/r}} \frac{1}{i}  \leq \sum_{\substack{i \in {\II}_{\leq x} \\ |c_i| \leq 1/r}} \frac{1}{i} = (1+o(1)) \cdot 2\mu([0,r^{-1}]) \cdot (C_2 \cdot g(x)).
\end{equation}
Estimating the second sum on the right-hand-side of~\eqref{dubbleD} is certainly more involved; the following Lemma~\ref{lem:shakeitoff} shows how this can be done by making explicit use of our assumption that $g(x) = \log x$ or $g(x) = \log \log x$.
\begin{lemma}\label{lem:shakeitoff}
For an initial string $S_b$ of digits in a given base $b \geq 2$ and an integer $r \geq 2$, we have that as $x \to \infty$,
\begin{eqnarray*}
	& & (1 + o(1)) \cdot \left(\log_b(1+S_b^{-1}) - \log_b(1+r^{-1})\right) \cdot (C_2 \cdot g(x)) \\
	& \leq & \sum_{\substack{i \in {\II}_{\leq x} \\ a_i \in  A(b,S_b) \\ |c_i| > 1/r}} \frac{1}{i} \\
	& \leq & (1 + o(1)) \cdot \left(\log_b(1+S_b^{-1}) + \log_b(1+r^{-1})\right) \cdot (C_2 \cdot g(x)) + K \cdot g(r),
\end{eqnarray*}
for some constant $K > 0$, possibly depending on the fixed parameters $b$, $S_b$, $m$, and $C_1$.
\end{lemma}
\begin{proof}
In what follows we shall give a proof of the upper bound; we omit the proof of the lower bound because it is analogous. Upon observing that we may write $$A(b,S_b) = \bigcup_{t = -\infty}^\infty \{x \in \RR :  |x| \in [S_b \cdot b^t , (S_b + 1) \cdot b^t)\},$$
we can split the desired sum as follows:
\begin{equation}\label{taytay}
\sum_{\substack{i \in {\II}_{\leq x} \\ a_i \in  A(b,S_b) \\ |c_i| > 1/r}} \frac{1}{i} = \sum_{t = -\infty}^\infty \sum_{\substack{i \in {\II}_{\leq x} \\ S_b \cdot b^t \leq |a_i| < (S_b + 1) \cdot b^t \\ |c_i| > 1/r}} \frac{1}{i}.
\end{equation}
When $t < 0$ and $|c_i| > 1/r$, the condition that $S_b \cdot b^t \leq |a_i| = C_1 \cdot i^m \cdot |c_i|\leq  (S_b + 1) \cdot b^t$ implies that $i \leq \left(\frac{S_b+1}{C_1 \cdot b}r\right)^{\frac{1}{m}}$, so the terms with $t < 0$ in~\eqref{taytay} can be bounded as follows:
$$\sum_{t = -\infty}^{-1} \sum_{\substack{i \in {\II}_{\leq x} \\ S_b \cdot b^t \leq |a_i| < (S_b + 1) \cdot b^t \\ |c_i| > 1/r}} \frac{1}{i} \leq \sum_{\substack{i \in \II \\ i \leq \left(C_2 \cdot \frac{S_b+1}{C_1 \cdot b}r\right)^{\frac{1}{m}}}} \frac{1}{i} \leq K \cdot g(r)$$
for some constant $K > 0$, possibly depending on the fixed parameters $m$, $b$, $S_b$, and $C_1$. We may now restrict our attention to the terms with $t \geq 0$ in~\eqref{taytay}; to bound these terms, we split the sum even further:
\begin{align*}
\sum_{t = 0}^\infty \sum_{\substack{i \in {\II}_{\leq x} \\ S_b \cdot b^t \leq |a_i| < (S_b + 1) \cdot b^t \\ |c_i| > 1/r}} \frac{1}{i} = \sum_{j = r}^{r^2 - 1}\sum_{t = 0}^\infty \sum_{\substack{i \in {\II}_{\leq x} \\ S_b \cdot b^t \leq |a_i| < (S_b + 1) \cdot b^t \\ j/r^2 < |c_i| \leq (j+1)/r^2}} \frac{1}{i},
\end{align*}
where we can afford to be loose about the order of summation because the contribution is $0$ for all but finitely many values of $t$. When $\frac{j}{r^2} \leq |c_i| \leq \frac{j+1}{r^2}$, the condition $S_b \cdot b^t \leq |a_i| = C_1 \cdot i^m \cdot |c_i| \leq (S_b + 1) \cdot b^t$ implies that $\left(\tfrac{S_b \cdot b^t}{C_1 \cdot (j+1)}r^2\right)^{\frac{1}{m}} \leq i \leq \left(\tfrac{(S_b+1) \cdot b^t}{C_1 \cdot j}r^2\right)^{\frac{1}{m}}$; this observation along with~\eqref{imfeelin22} and the condition that $\mu$ is symmetric allows us to make the following estimates:
\begin{align*}
\sum_{t = 0}^\infty \sum_{\substack{i \in {\II}_{\leq x} \\ S_b \cdot b^t \leq |a_i| < (S_b + 1) \cdot b^t \\ |c_i| > 1/r}} \frac{1}{i} & \leq \sum_{j = r}^{r^2 - 1}\sum_{t = 0}^\infty \sum_{\substack{i \in {\II}_{\leq x} \\ \left(\frac{S_b \cdot b^t}{C_1 \cdot (j+1)}r^2\right)^{\frac{1}{m}} \leq i \leq \left(\frac{(S_b+1) \cdot b^t}{C_1 \cdot j}r^2\right)^{\frac{1}{m}} \\ j/r^2 < |c_i| \leq (j+1)/r^2}} \frac{1}{i} \\
& = \sum_{j = r}^{r^2 - 1} (1 + o(1)) \cdot 2\mu([\tfrac{j}{r^2}, \tfrac{j+1}{r^2}]) \sum_{t = 0}^\infty \sum_{\substack{i \in {\II}_{\leq x} \\ \left(\frac{S_b \cdot b^t}{C_1 \cdot(j+1)}r^2\right)^{\frac{1}{m}} \leq i \leq \left(\frac{(S_b+1) \cdot b^t}{C_1 \cdot j}r^2\right)^{\frac{1}{m}}}} \frac{1}{i}.
\end{align*}
The condition that $i \leq x$ implies that for a given $j$, all terms with $\left(\frac{(S_b+1) \cdot b^t}{C_1 \cdot j}r^2\right)^{\frac{1}{m}} > x$, or equivalently $t > \log_b \frac{C_1 \cdot jx^m}{(S_b + 1)r^2}$, do not contribute to the sum. If we take $g(x) = \log x$, we have
\begin{align*}
\sum_{t = 0}^\infty \sum_{\substack{i \in {\II}_{\leq x} \\ S_b \cdot b^t \leq |a_i| < (S_b + 1) \cdot b^t \\ |c_i| > 1/r}} \frac{1}{i} & \leq (1 + o(1)) \cdot C_2 \cdot 2\mu([\tfrac{j}{r^2}, \tfrac{j+1}{r^2}]) \cdot \log_b \tfrac{C_1 \cdot jx^m}{(S_b + 1)r^2} \cdot \frac{\log\left((1+S_b^{-1})(1+j^{-1})\right)}{m} \\
& \leq  (1 + o(1)) \cdot C_2 \cdot \log \frac{C_1 \cdot x^m}{S_b + 1} \cdot \frac{\left(\log_b(1+S_b^{-1}) + \log_b(1+r^{-1})\right)}{m} \\
& = (1 + o(1)) \cdot \left(\log_b(1+S_b^{-1}) + \log_b(1+r^{-1})\right) \cdot (C_2 \cdot g(x)),
\end{align*}
which is the desired bound. If on the other hand we take $g(x) = \log \log x$, it is easy to check that the proof given in~\cite{ye} works \emph{mutatis mutandis} in our case, the only significant difference being the additional factor of $C_2$.
\end{proof}
We now proceed with the proof of~\ref{eq:yomama}. Applying Lemma~\ref{lem:shakeitoff} to bound the second sum on the right-hand-side of~\eqref{dubbleD} from above and below and combining the result with the our bound~\eqref{raiders} on the first sum yields that
\begin{eqnarray*}
	& & \left(1 + o(1)\right) \cdot (\log_b(1+S_b^{-1})-\log_b(1+r^{-1})) \cdot (C_2 \cdot g(x)) \\
	& \leq & \sum_{\substack{i \in {\II}_{\leq x} \\ a_i \in  A(b,S_b)}} \frac{1}{i} \\
	& \leq & (1 + o(1)) \cdot (\log_b(1+S_b^{-1})+\log_b(1+r^{-1}) + 2\mu([0,r^{-1}])) \cdot (C_2 \cdot g(x)) + K \cdot g(r).
\end{eqnarray*}
Let $\varepsilon \in (0, \log_b 2)$, and take $r > (b^\varepsilon-1)^{-1}$ so that $2\mu[(0,r^{-1})] < \varepsilon$. Further taking $x > \max\{r^{1/\varepsilon},\exp((\log r)^{1/\varepsilon})\}$, we find that
\begin{eqnarray*}
	& & (1+ o(1))\cdot (\log_b(1+S_b^{-1})-\varepsilon) \cdot (C_2 \cdot g(x)) \\
	& \leq &  \sum_{\substack{i \in {\II}_{\leq x} \\ a_i \in  A(b,S_b)}} \frac{1}{i} \\
	& \leq & (1 + o(1))\cdot \left(\log_b(1+S_b^{-1})+2\varepsilon+ \frac{K \cdot g(r)}{C_2 \cdot g(x)} \right) \cdot (C_2 \cdot g(x)) \\
	& \leq & (1 + o(1))\cdot \left(\log_b(1+S_b^{-1})+2\varepsilon + \frac{K}{C_2}\varepsilon\right) \cdot (C_2 \cdot g(x)),
\end{eqnarray*}
Taking $\varepsilon \to 0$, we obtain the desired result.

\section{Frobenius Traces of Newforms}\label{sadist}

We now apply the results of Theorems~\ref{thm:arithmetic} and~\ref{thm:logarithmic} to study the Frobenius traces of newforms. As the case of newforms without complex multiplication is studied in~\cite{ye}, we shall consider the case of newforms with complex multiplication; however, we note that the theorems stated in this section hold in both cases.

Given a newform $f \in S_k^{\text{new}}(\Gamma_0(N))$ of even weight $k \geq 2$ and trivial nebentypus on $\Gamma_0(N)$ that has complex multiplication by an order in a (necessarily imaginary quadratic) number field $K$, let $a_f(p)$ denote the trace of Frobenius of $f$ at $p$ for primes $p$. Recall that $a_f(p) = 0$ for primes $p$ if and only if $p$ is inert or ramified in $\OO_K$. Thus, we restrict our attention to the traces of Frobenius $a_f(p)$ at primes $p$ that split in $\OO_K$ (discarding the finitely many ramified primes). For convenience, let $\PP_K$ denote the set of primes that split in $\OO_K$, and for every $x > 0$, let ${\PP_K}_{\leq x} = \{p \in \PP_K : p \leq x\}$. By the Chebotarev Density Theorem, $\PP_K$ has arithmetic density, and hence logarithmic density, equal to $\frac 12$ in the set of all primes. Therefore, we have that
\[
\sum_{p \le x} \frac 1p = (1+o(1)) \cdot \log \log x \Rightarrow \sum_{\substack{p\in {\PP_K}_{\leq x}}} \frac 1p =\left(1+o(1)\right) \cdot \frac{1}{2} \cdot \log \log x,
\]
where the asymptotic on the left-hand-side of the above implication follows from the proof of Dirichlet's Theorem for primes in arithmetic progressions.

For each prime $p \in \PP_K$, let $\cos \theta_p = a_f(p)/\big(2p^{\frac{k-1}{2}}\big) \in [-1, 1]$; this is well-defined by the Hasse bound. Recall that a newform has complex multiplication by an imaginary quadratic field $K$ if and only if it comes from a Gr\"ossencharakter of $K$ (see Proposition 4.4 and Theorem 4.5 of~\cite{ribet}). So, by Hecke's equidistribution result~\cite{hecke} for the angles given by Gr\"ossencharakters of imaginary quadratic fields over $\QQ$, we have the following:
\begin{theorem}
	\label{thm:satotate}
    Let $f \in S_k^{\operatorname{new}}(\Gamma_0(N))$ be a newform of even weight $k \geq 2$ and trivial nebentypus on $\Gamma_0(N)$ that has complex multiplication by an order in a (necessarily imaginary quadratic) number field $K$. Then, for any subinterval $[A, B] \subset [-1, 1]$, we have that
    $$\lim_{x \to \infty} \frac{\#\{p \leq x : p \in \PP_K\, \operatorname{and}\, \cos \theta_p \in [A, B]\}}{\#\{p \leq x : p \in \PP_K\}} = \mu([A, B]),$$
where $\mu$ is the complex-multiplication analogue of the Sato-Tate measure, defined by   
\begin{equation}\label{rollinginthedeep}
d\mu = \frac{1}{\pi} \frac{dt}{\sqrt{1-t^2}}.
\end{equation}
\end{theorem}

Note that the measure $\mu$ defined in~\eqref{rollinginthedeep} has the desired convexity property. From the above discussion, it follows that the sequence $\{a_f(p) : p \text{ prime}\}$ fulfills the hypotheses of Theorems~\ref{thm:arithmetic} and~\ref{thm:logarithmic}. Thus, we obtain the following results as immediate corollaries:
\begin{theorem}
	Retain the setting of Theorem~\ref{thm:satotate}. The sequence $\{a_f(p)\}_{p \in \PP_K}$ is not arithmetically Benford in any base $b \ge 2$.
	\label{shizzle:arithmetic}
\end{theorem}
\begin{theorem}
	Retain the setting of Theorem~\ref{thm:satotate}.
	The sequence $\{a_f(p)\}_{p \in \PP_K}$ is logarithmically Benford in every base $b \ge 2$.
	\label{shizzle:logarithmic}
\end{theorem}
In summary, we have shown that the Benford's Law-type results on the Frobenius traces of newforms without complex multiplication proven in~\cite{ye} also hold for newforms with complex multiplication.

\section*{Acknowledgments}
\noindent This research was supervised by Ken Ono at the Emory University Mathematics REU and was supported by the National Science Foundation (grant number DMS-1250467). We would like to thank Ken Ono and Jesse Thorner for offering their advice and guidance and for providing many helpful discussions and valuable suggestions on the paper. We would also like to thank the anonymous referee for providing numerous helpful comments on an earlier draft of the paper.

\bibliographystyle{amsxport}
\bibliography{biblio3}

\end{document}